\documentclass[11pt,a4paper,twoside]{article}
\ifx\pdfpageheight\undefined
   \usepackage[dvips,colorlinks=true,linkcolor=blue,citecolor=red,%
      urlcolor=green]{hyperref}
   \usepackage[dvips]{graphicx}
   \makeatletter
   \edef\Gin@extensions{\Gin@extensions,.mps}
   \makeatother
\else
   \usepackage[pdftex]{graphicx}
   \usepackage[bookmarksopen=false,pdftex=true,breaklinks=true,%
      backref=page,pagebackref=true,plainpages=false,%
      hyperindex=true,pdfstartview=FitH,%
      colorlinks=true,%
      pdfpagelabels=true,linkcolor=blue,%
      citecolor=red,urlcolor=green,hypertexnames=false%
      ]%
   {hyperref}
\fi

\usepackage{a4wide}
\usepackage{verbatim}
\usepackage{proof}
\usepackage{latexsym}
\usepackage{amssymb}
\usepackage[utf8]{inputenc} 
\usepackage[T1]{fontenc}

\usepackage{stmaryrd}

\usepackage{mytheorems}
\newtheorem{context}[theorem]{Context}

\newcommand \ov[1]{\overline{#1}}
\newcommand \und[1]{\underline{#1}}
\newcommand \cmatrix[1]{\left[\matrix{#1}\right]}

\newcommand \NN{\mathbb{N}}
\newcommand \QQ{\mathbb{Q}}

\newcommand{\som}{\sum\nolimits}
\newcommand{\gothic}{\mathfrak}

\newcommand{\fI}{{\gothic I}}
\newcommand{\fJ}{{\gothic J}}
\newcommand{\fa}{{\gothic a}}
\newcommand{\fm}{{\gothic m}}
\newcommand{\fM}{{\gothic M}}

\newcommand{\bul}{^{\bullet}}

\newcommand\gen[1]{{\langle #1 \rangle}}
\newcommand\lrb[1] {\llbracket #1 \rrbracket}
\newcommand\vers[1]{\buildrel{#1}\over \longrightarrow }
\newcommand{\cro}[1]{\left[#1\right]}
\newcommand{\crac}[2]{\cro{\frac{#1}{#2}}}

\newcommand\In {\mathrm{I}_n}
\newcommand\Af {{A_{\lrb f}}}
\newcommand\Afn {A_{\lrb{f_1,\dots,f_n}}}

\newcommand{\mod}{\;\mathrm{mod}\;}
\newcommand{\rc}{\mathrm{c}}

\newcommand{\Jac}{\mathrm{Jac}}

\newcommand\junk[1]{}

\newcommand \Grandcadre[1]{%
\begin{center}
\begin{tabular}{|c|}
\hline
~\\[-3mm]
#1\\[-3mm]
~\\
\hline
\end{tabular}
\end{center}
}

\newcount\hh
\newcount\mm
\mm=\time
\hh=\time
\divide\hh by 60
\divide\mm by 60
\multiply\mm by 60
\mm=-\mm
\advance\mm by \time
\def\hhmm{\number\hh:\ifnum\mm<10{}0\fi\number\mm}
\def\hhhmm{\number\hh h.\,\ifnum\mm<10{}0\fi\number\mm}

\begin{document}
\markboth{Revisiting Zariski Main Theorem}{Alonso M., Coquand T., Lombardi H. 
}

\title{Revisiting Zariski Main Theorem from a constructive point of view}

\author{Alonso M. E., Coquand T., Lombardi H.}
\date{January 2016}

\maketitle

\noindent Note.

\noindent  
This paper appeared in Journal of Algebra {\bf 406}, (2014), 46--68

\smallskip\noindent  Here, we have fixed two typos.

\smallskip \noindent At the end of the proof of Proposition 4.8, we  write

\hspace{1em}and $h(T)=T^{N+q}(T-1)$

\noindent instead of:  $h(T)=T^{N-1}(T-1)$

\smallskip \noindent In the proof of Lemma 4.9. 
In line 7 of the proof we write: 

\hspace{1em}We have $q(T)\in T^{N+1}+\fM A[T]$

\noindent instead of:  We have $q(T)\in T^{N+1}(T-1)+\fM A[T]$

\newcommand\hum[1]{}

\begin{abstract}
This paper deals with the Peskine version of Zariski Main Theorem published in 1965  and discusses some applications. 
It is written in the style of Bishop's constructive mathematics. Being constructive, each proof in this paper can be interpreted as
an algorithm for constructing explicitly the conclusion from the hypothesis.
The main non-constructive argument in the proof of Peskine is the use of minimal prime ideals.  Essentially we substitute this point by two dynamical arguments; one about 
 gcd's, using  {subresultants}, and another
 using our  notion of  {strong transcendence}.   
In particular we obtain algorithmic versions for the Multivariate Hensel Lemma
and the structure theorem of quasi-finite algebras.
\end{abstract}

\noindent {\bf\large Keywords.} Zariski Main Theorem, Multivariate Hensel Lemma, Quasi finite algebras, Constructive Mathematics

\tableofcontents

\section{Introduction}

The paper is
written in the style of Bishop's constructive mathematics,
i.e.\ mathematics with intuitionistic logic (see \cite{BB,BR,LQPTF,Richman}).

A partial realization of Hilbert's program has recently proved successful
in commutative algebra, see e.g., \cite{ALP,coq:seminormal,coq:valspace,coq:generating,CQ2012,DLQS,LQPTF,yengui:maximal} and \cite{CLS} with references therein, and this paper is a new piece of realization of this program.

\medskip 
We were mainly interested in an algorithm for the Multivariate Hensel Lemma
(MHL for short).
Let us see what is the aim of the computation on a simple example.

We consider the local ring $A=\QQ[a,b]_S$, $S=1+\gen{a,b}A$.
We take the equations
$$-a + x   + bxy + 2bx^2      = 0,~~~~~~ -b + y +ax^2 + axy + by^2 = 0$$
and we want to compute a solution of the system  $(\xi,\zeta)\equiv 0 \mod \fM$
in the henselization of $A$. In other words, we have to find a Hensel equation $f(U)\in A[U]$ (i.e.\ $f$ monic, $f(0)\in\gen{a,b}$ and $f'(0)\notin\gen{a,b}$) such that, when adding the Hensel zero $u$ of $f$ to $A$
we are able to compute  $\xi$ and $\zeta\in A[u]_{1+\gen{a,b,u}A[u]}$.

Surprinsingly there is no direct proof of the result. Moreover elementary elimination techniques do not work on the above example. So we have to rely on the proof of MHL via the so called Zariski Main Theorem (ZMT for short), as for example in \cite{Ray}.   Note that there are many versions of ZMT
(e.g.\ \cite{EGA4,Zar2}) and we are interested in the ZMT \`a la Peskine as in \cite{Ray}.   

We will give a solution of the above example in section \ref{subsecMHLexample}.

\medskip    This paper deals  with the Peskine proof of ZMT published in 1965  \cite{Peskine}  and discusses some applications.
Peskine statement is purely algebraic avoiding any hypothesis of noetherianity. 
The argument we give for Theorem \ref{main2bis} follows 
rather closely Peskine's proof.
The main non-constructive argument in the proof of Peskine is the use of minimal prime ideals. Note that the existence of minimal prime ideals in commutative rings is known to be equivalent to Choice Axiom. Essentially we substitute this point by two dynamical arguments; one about 
 gcd's, using \emph{subresultants}, section \ref{subsecCrulem}, proof of Proposition \ref{13.7}, and another
 using our  notion of \emph{strong transcendence}, section \ref{secstrongtrans}
(in classical mathematics: to be transcendent over all residual fields).

In sections \ref{secHLR} and \ref{subsecThmain3}, we give a constructive treatment of two classical applications of ZMT: the Multivariate Hensel Lemma, and structure theorem of quasi finite algebras.

Being constructive, each proof in this paper can be interpreted as
an algorithm for constructing explicitly the conclusion from the hypothesis.  

\begin{theorem}\label{main2} \emph{(ZMT \`a la Peskine, particular case)}\\
Let $A$ be a ring, $\fM$ a detachable maximal ideal of $A$ and $k=A/\fM$. If $B=A[x_1,\dots,x_n]$ 
is an extension of $A$ 
such that $B/\fM B$ is a finite $k$-algebra then there exists $s\in 1+\fM B$ such that
$s$, $sx_1$, \dots, $sx_n$ are integral over $A$.
\end{theorem}

In \cite{PeskineBook}  an equivalent formulation (Proposition 13.4)
of Peskine version of the Zariski Main Theorem can be written as the following
lemma. 

\medskip \noindent 
{\bf Proposition.} {\it Let $(A,\fM)$ be a residually discrete local ring and $k=A/\fM$. 
If $B = A[x_1,\dots,x_n]$ is an extension of $A$ such that
$\fM B\cap A = \fM$, $A$ is integrally closed in $B$ and $B/\fM B$
has a nontrivial zero-dimensional component as a $k$-algebra, 
 then $B=A$.
}
  
\medskip The last hypothesis can be given in a concrete way: there exists
an idempotent $e$ of $B/\fM B$ such that $(B/\fM B)[1/e]$ is a nontrivial
 finite $k$-algebra. This means that the residual variety has at least one isolated point.

\medskip 
The following corollary of Theorem \ref{main2}  is a weakened form of the previous proposition.

\begin{corollary}\label{varmain2}
Let $(A,\fM)$ be a residually discrete local ring and $k=A/\fM$. 
If $B = A[x_1,\dots,x_n]$ is an extension of $A$ such that
$\fM B\cap A = \fM$, $A$ is integrally closed in $B$ and $B/\fM B$ is a finite $k$-algebra then $B=A$.
\end{corollary}

\begin{proof}
By Theorem \ref{main2} we find $s\in A$ such that $s\in1+\fM B$ and
$sx_1,\dots,sx_n\in A$. We have then $s-1\in A\cap \fM B = \fM$
and hence $s$ is invertible in $A$. Hence $x_1,\dots,x_n$ are in $A$ and $B = A$.
\end{proof}

\noindent \emph{Remark.}
The hypothesis that $A$ is integrally closed in $B$ is necessary, even if we weaken the conclusion to ``$B$ is finite over $A$''.
Let $A$ be a DVR with $\fM=pA$, the ring $B=A\times A[1/p]$ is finitely generated over $A$, $\fM B=\gen{(p,1)}$ and $B/\fM B=A/\fM$, but $B$ is not finite over $A$.
If $A'$ is the integral closure of $A$ in $B$, we cannot apply Corollary
\ref{varmain2} with $(A',\fM A')$ replacing $(A,\fM)$ because $\fM A'$
 is not a maximal ideal of $A'$ (in fact $A'\simeq A\times A$).

\medskip 
In fact we shall prove a slightly more general version
of Theorem \ref{main2}, without assuming $\fM$
to be a detachable maximal ideal.

\begin{theorem}\label{main2bis}\emph{(ZMT \`a la Peskine, variant)}\\
Let $A$ be a ring with an ideal $\fI$
and $B=A[x_1,\dots,x_n]$ be an extension of $A$ such that $B/\fI B$ is a finite $A/\fI$-algebra,  then there exists $s\in 1+\fI B$ such that
$s$, $sx_1$, \dots, $sx_n$ are integral over~$A$.
\end{theorem}

\noindent \emph{Remark.} In fact, the hypothesis that the morphism $A\to B$ is injective is not necessary: it is always possible to replace $A$ and $\fI$ by their images in $B$, and the conclusion remains the same. 

\begin{corollary} \label{cormain2bis}
Let $A$ be a ring with an ideal $\fI$
and $B=A[x_1,\dots,x_n]$ be an extension of $A$  such that $B/\fI B$ 
is a finite $A/\fI$-algebra,  then there exists  a 
finite extension $C$ of $A$ inside $B$ and $s\in C\cap 1+\fI B$ such that $C[1/s]=B[1/s]$. 
\end{corollary}
%
\begin{proof}
Take $C=A[s,sx_1,\dots,sx_{n}]$.
\end{proof}

We shall also give a proof of the following ``global form'' of Zariski Main Theorem.

\medskip\noindent  
{\bf Theorem \ref{main3}} 
{\it  \emph{(ZMT \`a la Raynaud, \cite{Ray})}\\
Let $A\subseteq B = A[x_1,\dots,x_n]$ be rings  
 such that the inclusion morphism $A\to B$ is zero dimensional
 (in other words, $B$ is quasi-finite over $A$).
Let $C$ be the integral closure of $A$ in $B$.
Then  there exist elements $s_1,\dots,s_m$ in $C$, comaximal in $B$, such that
all  $s_ix_j\in C$. 
\\In particular for each $i$, $C[1/s_i]=B[1/s_i]$. 
Moreover  letting $C'=A[(s_i),(s_ix_j)],$ which is finite over $A$,
 we get also  $C'[1/s_i]=B[1/s_i]$ for each $i$.
}

\bigskip  We give now the plan of the paper.

\medskip In section \ref{secPeskLem} we give some preliminary results
and the proof of a Peskine ``crucial lemma''.

\medskip In section \ref{secProofZMT} we give the constructive proof for Theorem \ref{main2bis}.


\medskip In section \ref{secHLR} we give a constructive proof for the Multivariate Hensel Lemma (Theorem \ref{MHL}). A usual variant is 
the following  corollary. 

\medskip\noindent  
{\bf Corollary \ref{corMHL}} 
{\it Let  $(A,\fm)$ be a Henselian local ring. Assume that a polynomial
system
$
(f_1 , \dots , f_n)  
$ in $A[X_1,\dots,X_n]$
has residually a simple zero at $(0,\dots,0)$. Then the system has a (unique) solution in $A^{n}$ with coordinates in $\fm$.}

\medskip Section \ref{subsecThmain3} is devoted to structure theorem
of quasi-finite algebras: we give a proof of Theorem \ref{main3},
moreover Proposition~\ref{propdefiquasifini} explains the constructive content of the hypothesis in Theorem \ref{main3}.

\medskip \noindent {\bf Acknowledgements}. First and third authors  
are partially supported by Spanish GR MTM2011-22435. Third author thanks the
 Computer Science and Engineering Department at University of Gothenburg for several invitations. This article has been discussed in the course of a researching stay of the first author at the Department of Mathematics of the University of Franche-Comt\'e. 
She  thanks the Department for its  kind invitation.

\section{Peskine crucial lemma}
\label{secPeskLem}

In this section we give a constructive proof of a crucial lemma in the proof of Peskine. This is Proposition \ref{13.7} in the following.

\subsection{Basic tools for computing integral elements}


 Let $R\subseteq S$ be rings and let ${\fI}$ be an ideal of $R$. We say that $t\in S$
is {\em integral over} ${\fI}$ if and only if it satisfies a relation
$t^n+a_1t^{n-1}+ \cdots+a_n = 0$ with $a_1,\dots,a_n$ in ${\fI}$. The {\em integral closure}
of ${\fI}$ in $S$ is the ideal of elements of $S$ that are integral over ${\fI}$.


\goodbreak
\begin{lemma}\label{atiyah} (Lying Over, concrete form)
\begin{enumerate}
\item If $S$ is integral over $R$ then the integral closure of ${\fI}$ in $S$ is $\sqrt{{\fI}S}$.

\noindent 
As a consequence $\sqrt{{\fI}S}\cap R=\sqrt{{\fI}}$.
\item If $S$ is integral over $R$ and $1\in \gen{b_1,\dots,b_m}S$
then $1\in \gen{b_1,\dots,b_m}R[b_1,\dots,b_m]$.
\end{enumerate}

\end{lemma}

\begin{proof}
\emph{1.} See \cite{Atiyah} Lemma 5.14.

\noindent \emph{2.} Use item \emph{1} with $R'=R[b_1,\dots,b_m]$
and $\fI=\gen{b_1,\dots,b_m}R[b_1,\dots,b_m]$.
\end{proof}

\noindent {\bf Algorithm:} Let $x$ be in $\sqrt{{\fI}S}$: $x^n=\sum_{k=1}^{p}a_ks_k$, with $a_k\in {\fI}$ and $s_k\in S$. Let $1=s_1,\dots,s_m$ be generators of 
$R_1=R[s_1,\dots,s_p]$ as an $R$-module. The multiplication by $x^n$ in $R_1$ 
is expressed on $s_1,\dots,s_m$ by a matrix $M_{x^n}$ with coefficients in ${\fI}$.
The characteristic polynomial of $M_{x^n}$ is $P(T)=T^m+\sum_{k=0}^{m-1}b_kT^k$
with $b_k$'s $\in {\fI}$, and $P(x^n)=P(M_{x^n})(1,0,\dots,0)=0$. 

\begin{definition} \label{deficontent}
We denote $\rc_X(g)$ (or $\rc(g)$) the ideal of $R$ generated by the coefficients 
of $g\in R[X]$
($\rc_X(g)$ is called the $X$-content ideal of $g$ in $R$).
\end{definition}

\begin{lemma}\label{kronecker} \emph{(Kronecker)}

\noindent  Let $Z\subseteq R$ where $Z$
is the subring generated by $1$.
\begin{enumerate}
\item (simple form) If $f(X)=X^k+a_1X^{k-1}+ \cdots+a_k$ divides $X^n+b_1X^{n-1}+ \cdots + b_n$ in $R[X]$ then
$a_1,\dots,a_k$ are integral over $b_1,\dots, b_n$
(more precisely they are integral over the ideal generated by $b_1,\dots, b_n$
in  $Z[b_1,\dots, b_n]$).
\item (general form) If $fg=h=\sum_{j=0}^nc_jX^j$ in $R[X]$, $a$ a coefficient of $f$
and $b$ a coefficient of $g$ then $ab$ is integral over the ideal generated by $c_0,\dots, c_n$
in  $Z[c_0,\dots, c_n]$.
\item (Gauss-Joyal) If $fg=h=\sum_{j=0}^nc_jX^j$ in $R[X]$ then $\rc(f)\rc(g)\subseteq \sqrt{\rc(h)}$.
\end{enumerate}

\end{lemma}

\begin{proof} \emph{1.}
Considering the splitting algebra of $f$ over $R$, we can assume $X^k+a_1X^{k-1}+ \cdots+a_k = (X-t_1)\dots (X-t_k)$. We have
then $t_1,\dots,t_k$  integral over $b_1,\dots, b_n$ and hence also
$a_1,\dots,a_k$ since they are (symmetric) polynomials in $t_1,\dots,t_k$.

\noindent \emph{2.} This is deduced from \emph{1} by homogeneization arguments.

\noindent \emph{3.} This is an immediate consequence of \emph{2.}
\end{proof}

\begin{lemma}\label{basicEmmanuel}
If $R\subseteq S$ and $t\in S$ satisfies an equation $a_nt^n+ \cdots + a_0 = 0$
with $a_0,\dots,a_n\in R$ then $a_nt$ is integral over $R$.
\end{lemma}


\begin{lemma}\label{Emmanuel}\emph{(see \cite{Hallouin})}
Let $R\subseteq S$  and $x\in S$ satisfies an equation $P(x)=a_nx^n+ \cdots + a_0 = 0$
with $a_0,\dots,a_n\in R$. We take 
$$
  u_n = a_n,\; u_{n-1} = u_nx+a_{n-1},\;\dots\dots\;,\;u_0 = u_1x + a_0=0
$$
We get the following results.
\begin{enumerate}
\item  $u_n,\dots,u_0$ and $u_nx,\dots,u_0x$ are integral over $R$ and
$\gen{u_0,\dots,u_n} = \gen{a_0,\dots,a_n}$ as ideals of $R[x]$.
\item Let ${\fI}$ be an ideal of $R$ s.t.\  
 $1\in \gen{a_0,\dots,a_n}R[x] \mod {\fI}R[x]$ and
 $x\mod {\fI}$ is integral over $R/{\fI}$ 
 then there exists $w\in1+R[x]$ s.t.\ 
$w$ and $wx$ are integral over $R$.
\end{enumerate}

\end{lemma}

\begin{proof}
\emph{1}. Lemma \ref{basicEmmanuel} shows that $a_nx=u_nx$ is integral over $R$.
 It follows that $u_{n-1} = u_nx + a_{n-1}$ is integral over $R$.
We have then
$$u_{n-1}x^{n-1} + a_{n-2}x^{n-2} + \cdots + a_0 = 0$$
so that, again by Lemma \ref{basicEmmanuel}, $u_{n-1}x$ is integral over $R[u_{n-1}]$ and so over $R$.
In this way, we get that $u_n$, $u_nx$, $u_{n-1}$, $u_{n-1}x$, \dots, $u_1x$, $u_0= 0$ are all integral
over $R$.

\smallskip \noindent \emph{2}. Let $R'$ be the image of $R[\und{u}]=R[u_0,\dots,u_n]$
in $R[x]/{\fI}R[x]$. So $R'\subseteq R[x]/{\fI}R[x]$ with $R[x]/{\fI}R[x]$ integral over $R'$. Item \emph{1.} shows that $1\in \gen{u_0,\dots,u_n} \mod {\fI}R[x]$.
Lying Over item \emph{2} gives $1=\sum_{i=0}^nu_ig_i(\und{u})\mod {\fI}R[x]$ with $g_i(\und{u})$'s $\in R[\und{u}]$. Let $w=\sum_{i=0}^nu_ig_i(\und{u})$, then $w$ and $wx$ are clearly integral over  $R$.
\end{proof}

\begin{lemma} \label{comp} \label{varcomp}~ 
\begin{enumerate}
\item  If $t$ is integral over $R[x]$ and $p(x)$ is a monic polynomial in 
$R[x]$ such that $tp(x)$ is in $R[x]$ then there exists $q$ in $R[x]$
such that $t-q$ is integral over $R$.

\item  If $t$ is integral over $R[x]$ and $p(x)= a_kx^k + \cdots + a_0$ is a polynomial in 
$R[x]$ such that $tp(x)$ is in $R[x]$ then there exists $q$ in $R[x]$ and $m$ 
such that $a_k^mt-q$ is integral over $R$.
\end{enumerate}

\end{lemma}

\begin{proof}
\emph{1.} We write $tp = r(x)$ in $R[x]$. We do the Euclidian division of $r(X)$
by $p(X)$ and get $r = pq + r_1$. We can then write $(t-q)p = r_1$. This
shows that we have $p = (t-q)^{-1}r_1$ in $R[(t-q)^{-1}][x]$ and hence
that $x$ is integral over $R[(t-q)^{-1}]$. Since $t-q$ is integral over $R[x]$
we get that $t-q$ is integral over $R[(t-q)^{-1}]$ and hence over $R$.

\noindent 
\emph{2.} We have an equation for $t$ of the form $t^n +p_1(x)t^{n-1}+ \cdots+p_n(x) = 0$.
Let $\ell $ be the greatest exponent of $x$ in this expression. By multiplying by
$a^\ell $ we get an equality of the form
$$
a^\ell  t^n + q_1(ax) t^{n-1} + \cdots + q_n(ax) = 0
$$
and hence, by Lemma \ref{basicEmmanuel}, $a^\ell t$ is integral over $R[ax]$.
Sowe we have $\ell$ such that $a^\ell t$ is integral over $R[ax]$ for all $a\in R$.
\\
We write $tp(x) = r(x)$ and by multiplying by a suitable power of $a_k$ we get an 
$ta_k^m P(a_k x) \in R[a_k x]$ with $m\geq \ell$ and $P$ monic. We
can then apply item \emph{1.}
\end{proof}

\begin{corollary}\label{corvarcomp}
If $t$ is integral over $R[x]$ and $R$ is integrally closed in $R[x,t]$ and
$t(a_kx^k+ \cdots+a_0) \in R[x]$ then there exists $m$ such that $a_k^mt\in R[x]$.
\end{corollary}

Next lemma is a kind of glueing of integral extensions.

\begin{lemma}\label{int}
Let $R\subseteq S$ and $x,t,y,s\in S$. If $t,ty$ are integral over $R[x]$ and $s,sx$ integral over $R$ then for $N$
big enough and $w=s^Nt$ the elements $w,wx,wy$ are integral over $R$.
\end{lemma}

\begin{proof}
We write $t^k + a_1(x)t^{k-1} + \cdots + a_k(x) = 0$ and
$t^\ell y^\ell  + b_1(x) t^{\ell -1}y^{\ell -1} + \cdots + b_\ell  = 0$. Let $x^d$ be the highest
power of $x$ that appears in these expressions. We have that $s^dt$ and
$s^dty$ are integral over $s,sx$ and so over $R$, and we take $N=d+1$.
\end{proof}

\subsection{Strong transcendence}\label{secstrongtrans}

Let $ D$ be a $C$-algebra and $x\in$ $D$. We say that \emph{$x$ is strongly transcendent over $C$ in $D$} if for all
$u\in D$ and $c_0,\dots,c_k\in C$ such that $u(c_0+ \cdots+c_kx^k) = 0$, we have
$uc_0 = \cdots = uc_k = 0$ (each time it is needed,  $c_i$ stands for  the image of $c_i$ in $D$).

Note that the definition strongly depends on $C$ {\bf and} $D$.
Moreover from an equality $c_0+ \cdots+c_kx^k = 0$ in $D$ we deduce only that
$c_0 = \cdots = c_k = 0$ in $D$.

\penalty-2500
\begin{lemma}\label{trans1}~
\begin{enumerate}
\item If  $D$ is a  $C$-algebra, $x$  strongly transcendent over $C$ in $D$ and
$V$  a monoid of $D$, then $x$ is strongly transcendent over $C$ in $D_V$.
\item If  $D$ is a  $C$-algebra and $x$ is strongly transcendent over $C$
in $D$ and $a\in C$,
then $x$ is strongly transcendent over $C[1/a]$ in $D[1/a]$.
\end{enumerate}
\end{lemma}

\begin{lemma}\label{integral}
If $u,x\in D$, $D$ a reduced $C$-algebra, $x$ strongly transcendent over $C$
 in $D$ and $u,ux$ are integral over $C$ then $u=0$.
\end{lemma}

\begin{proof}
We have $Q(ux)=(ux)^\ell + c_1 (ux)^{\ell-1} + \cdots + c_\ell = 0$ and $P(u)=u^m+a_1u^{m_1}+ \cdots + a_m = 0$ for some $c_1,\dots, c_\ell,a_1,\dots,a_m$ in $C$.
So $\mathrm{Res}_U(P(U),Q(Ux))=V(x)$ is a polynomial with constant coefficient $c_\ell^m$ and leading coefficient $\pm a_m^{\ell}$. Since $V(x)=0$, $x$ is transcendent over $C$ in $D$, and $D$ is reduced,
it follows that we have $c_\ell = a_m= 0$ in $D$. 
We get in $D$
$$(ux)Q_1(ux)=0 \;\;\hbox{ with  }\; Q_1(T)=(Q(T)-c_\ell)/T.
$$
Now we consider the reduced ring $ D_1=D[1/(ux)]$. In this ring $x$ is strongly transcendent over $C$.
We have in $D_1$  
$$Q_1(ux)=0 \hbox{ with  } Q_1(0)=c_{\ell-1} \;\; \hbox{ and }\;P(u)=0.$$
So $c_{\ell-1}=0$ in $D_1$. 
Similarly we deduce $c_{\ell-2}=\cdots=c_1=0$ in $D_1$.
So $ux=0$ in $D_1$ and finally $ux=u=0$ in $D$.
\end{proof}

\begin{lemma}\label{transcendent}
If  $D$ is a reduced $C$-algebra and $x$ is strongly transcendent over $C$
in $D$ and  $C_1\subseteq D$ and $C_1$ is integral over $C$ then 
$x$ is
strongly transcendent over $C_1$  in $D$. 
\end{lemma}
\begin{proof}
Assume an equality   $u(c_0x^k+ \cdots+c_k) = 0$ with $u\in D$ and $c_i$'s in $C_1$. Passing  to $D'=D[1/u]$ we get $c_0x^k+ \cdots+c_k = 0$ with
 $c_i$'s integral over $C$. So $c_0x$ is integral over $C_1$, and thus over $C$ too.
So~$c_0$ and $c_0x$ are integral over $C$. By Lemma \ref{integral}
$c_0=0$ in $D'$, so $c_0u=0$ in $D$. We finish by induction on $k$. 
\end{proof}
%

\subsection{Crucial lemma}
\label{subsecCrulem} 

\begin{context} \label{context1}
We fix now the following context, which comes from Corollary \ref{corvarcomp}: $t$ integral over $R[x]$ of degree $n$ and $R$ integrally closed in $S=R[x,t]$.
We define $J = (R[x]:S)$.  
\end{context}

\begin{lemma}\label{trivial} \emph{(Context \ref{context1})}\\
If $u\in S$ we have $u\in J$ if and only if $u,ut,\dots,ut^{n-1}\in R[x]$.
\end{lemma}

\begin{proof}
This is clear since all elements of $S$ can be written $q_{n-1}(x)t^{n-1}+ \cdots+q_0(x)$.
\end{proof}

\begin{lemma}\label{pre13.5} \emph{(Context \ref{context1})}\\
If $u\in S$ and $a_0,\dots,a_k\in R$ and $u(a_0+ \cdots+a_kx^k)\in J$, then there exists $m$ such that $ua_k^m\in J$.
\end{lemma}

\begin{proof}
We have by Lemma \ref{trivial}
$$
(a_0+ \cdots+a_kx^k)u, (a_0+ \cdots+a_kx^k)ut, \dots, (a_0+ \cdots+a_kx^k)ut^{n-1}\in R[x].
$$
All elements $ut^j$ are integral over $R[x]$ and $R$ is integrally closed in $R[x,ut^j]$.
Hence by Corollary \ref{corvarcomp} we find $m$ such that $a_k^m ut^j \in A[x]$.
\end{proof}

 We consider now the radical $\sqrt {JS}$ of $J$ {\em in} $S$.

\begin{corollary}\label{13.5} \emph{(Context \ref{context1})}\\
If $u\in S$ and $a_0,\dots,a_k\in R$ and $u(a_0+ \cdots+a_kx^k)\in \sqrt{JS}$, then 
$ua_0,\dots,ua_k\in \sqrt{JS}.$
\end{corollary}

\begin{proof}
We have $\ell $ such that $u^\ell (a_0+ \cdots+a_kx^k)^\ell \in J$. By Lemma \ref{pre13.5} we have
$m$ such that $u^\ell (a_k^\ell )^m\in J$ and hence $ua_k\in \sqrt{JS}$. It follows that
$ua_kx^k \in \sqrt{JS}
$ and so $u(a_0+ \cdots+a_{k-1}x^{k-1})\in \sqrt{JS}
$ and we get
successively $ua_{k-1},\dots,ua_0\in \sqrt{JS}
$.
\end{proof}

\medskip Summing up previous results in Context  \ref{context1} 
and using the notion of strong transcendence.
\begin{proposition}\label{quotient}
Assume $S=R[x,t]$ with $t$ integral over $R[x]$ and $R$ is integrally closed in $S$.
We take $J=(R[x]:S)$. If we take $D = S/\sqrt{JS}
$ and $C=R/R\cap \sqrt{JS}
$, then
$D=C[x,t]$ is a reduced ring  with a subring $C$ such that $t$ is integral over
$C[x]$ and $x$ is strongly transcendent over $C$  in $D$.
\end{proposition}
\begin{proof}
Clear. The last assertion comes from Corollary \ref{13.5}.
\end{proof}
%

\begin{proposition} \label{lemDtrivial}
Assume  that $D=C[x,t]$ is a reduced ring  with a subring $C$ such that $t$ is integral over
$C[x]$ and $x$ is strongly transcendent over $C$ in $D$. Let ${\fI}$ be an ideal of $C$ such that $tx\in\sqrt{{\fI}D}$.
Then $t\in\sqrt{{\fI}D}$. \\
Equivalently, if $D_U$ is the localization of $D$ at the monoid $U=t^\mathbb{N}+{\fI}D$, then
$D_U$ is a trivial ring.
\end{proposition}

The proof is given after the crucial lemma.

\begin{proposition}\label{13.7} \emph{(crucial lemma)}\\
If $S = R[x,t]$ and $R$ is integrally closed in $S$ and $t$ is integral
over $R[x]$ and ${\fI}$ ideal of $R$ such that $tx\in\sqrt{{\fI}S}$
then $t\in\sqrt{{\fI}S}$ mod.\ $\sqrt{JS}
$ where $J = (R[x]:S)$.
\end{proposition}

\begin{proof}
This follows from Propositions \ref{quotient} and  \ref{lemDtrivial}.
\end{proof}

\medskip \noindent  {\it Here begins the proof of Proposition \ref{lemDtrivial}}.

\noindent Since $t$ is integral over $C$ we get a $T$-monic polynomial $P(x,T)$ in $C[x][T]$ s.t.\ $P(x,t)=0$. As  $tx\in\sqrt{{\fI}D}$  by Lying Over
we get  a polynomial
$$Q(X,T) = X^nT^n + \mu_1(X) X^{n-1}T^{n-1} + \cdots + \mu_n(X) \hbox{ with }\in \mu_i(X)\in {\fI}C[X]$$ 
s.t.\ $Q(x,t)=0$.
 We need now to prove Lemma \ref{gcd}.

\begin{lemma}\label{gcd}
Assume $C_1\subseteq D_U$, that $x$ is transcendent over $C_1$ and
that $G(x,T) = T^k + b_1(x) T^{k-1} + \cdots + b_k(x)$ divides $Q(x,T)$, with
$b_1(x),\dots,b_k(x)\in C_1[x]$ and $G(x,t) = 0$. Then $D_U$ is a trivial ring.
\end{lemma}

\begin{proof}
Since $x$ is transcendent over $C_1$ we have that 
$G(X,T) = T^k + b_1(X)T^{k-1} + \cdots + b_k(X)$ divides $Q(X,T)$.
By taking $T = X^N$ we see that 
$X^{Nk} + b_1(X)X^{N(k-1)} + \cdots + b_k(X)$ 
divides
$X^nX^{Nn} + \mu_1(X) X^{n-1}X^{N(n-1)} + \cdots + \mu_n(X)$.
If $N$ is big enough we can apply Lemma \ref{kronecker} and conclude
that all coefficients of $b_1(X),\dots,b_k(X)$ are integral over ${\fI}$.
Since $G(x,t)=t^k + b_1(x)t^{k-1} + \cdots + b_k(x) = 0$ it follows that $t$
is integral over ${\fI}C[x]$, and so $D_U$ is a trivial ring.
\end{proof}

\noindent  
We consider the ring $D_U$, we compute the subresultants of $P(x,T)$ and $Q(x,T)$ in $C[x][T]$ and we show
that they are all $0$ in $D_U$, i.e.\ $P(x,T)$ has to divide $Q(x,T)$ in $D_U[T]$. 

\noindent  
The conclusion follows then from Lemma \ref{gcd} with $C_1$ the image of $C$ in
$D_U$ and $G=P$.

\noindent We use results about subresultants given in Lemma \ref{lemSoures}
(for the general theory of subresultants, see \cite[Chapter 4]{BPR})
We consider one such subresultant $s_0(x) T^\ell  + c_1(x) T^{\ell -1} + \cdots + c_\ell (x)$
assuming that all previous subresultants have been shown to be $0$. 
We can assume $s_0(x)$ to be invertible, replacing $D_U$ by $D_U[1/s_0]$.
We let $a$ be the leading coefficient of $s_0(x)$ and we show $a=0$.
We write $b_i(x) = c_i(x)/s_0(x)$. Since  $T^\ell  + b_1(x)T^{\ell -1} + \cdots + b_\ell (x)$ divides
$P(x,T)$ we have that $b_1(x),\dots,b_\ell (x)$ are integral over $C[x]$ by Lemma \ref{kronecker}. By 
Lemma \ref{comp},
$b_1(x),\dots,b_\ell (x)$ are in $C_1[1/a][x]$ with $C_1$ integral over $C$. By
Corollary \ref{transcendent} and Lemmas \ref{gcd} and \ref{trans1}, we have $1=0$ in $D_U[1/a]$
and hence $a=0$ in $D_U$.

\medskip \noindent  {\it Here the proof of Proposition \ref{lemDtrivial}
 is finished.}

\begin{lemma} \label{lemSoures}
Let $A$ be a reduced ring, $f\in A[X]$ a monic polynomial of degree $d$, $g\in A[X]$
and $\delta$ a bound for the degree of $g$.
Let $j<d$ a nonnegative integer. The subresultant of $f$ and $g$ in degree $j$, denoted  $\mathrm{Sres}_{j,X,d,\delta}(f,g)=Sr_{j}(X)$ is a well defined polynomial of degree $\leq j$: it does not depend on $\delta$. We let $Sr_d=f$. Let us denote $s_j$ the coefficient 
of $X^j$ in $Sr_{j}(X)$. Then we have:
\begin{enumerate}
\item $Sr_{j}(X)$ belongs to the ideal $\gen{f,g}$ of $A[X]$ ($0\leq j\leq d$).
\item Let $\ell>0$, $\ell\leq d$. If $s_k=0$ for $k<\ell$ and $s_\ell$ is invertible, then:

\noindent -- $Sr_{k}(X)=0$ for $k<\ell$.

\noindent -- $Sr_{\ell}(X)$ divides $f(X)$ and $g(X)$ in $A[X]$.
\end{enumerate}
\end{lemma}
%
\begin{proof} \emph{1.} This is a classical result.

\noindent \emph{2.}
Since the results are well known  when $A$ is a field,
the lemma follows by using the formal Nullstellensatz.
\end{proof}
%

\section{Proof of ZMT} \label{secProofZMT}
\label{subsecThmain2}

It is more convenient for a proof ``by
induction on $n$'' to use the following version \ref{main2ter}. 

\begin{theorem}\label{main2ter}\emph{(ZMT \`a la Peskine, general form, variant)}\\
Let $A$ be a ring with an ideal $\fI$
and $B$ be a finite extension of $A[x_1,\dots,x_n]$ such that 
$B/\fI B$ is a finite $A/\fI$-algebra,  then there exists 
$s\in 1+\fI B$ such that $s,sx_1,\dots,sx_n$ are integral over~$A$.
\end{theorem}

Here, the precise hypothesis is $A\subseteq A[x_1,\dots,x_n]\subseteq B$,
with $B$ finite over $A[x_1,\dots,x_n]$.  
Clearly
Theorems \ref{main2bis} and \ref{main2ter} are equivalent.

\subsubsection*{Case $n=1$}

\begin{proposition}\label{3.1}
Let $A$ be a ring with an ideal $\fI$
and $B$ be a finite extension of $A[x]$ such that $B/\fI B$ is a finite $A/\fI$-algebra,  then there exists $s\in 1+\fI B$ such that
$s,sx$ are integral over~$A$.
\end{proposition}
\begin{proof}
Let $f(X)\in A[X]$ a monic polynomial s.t.\ $f(x)\in\fI B$.
By Lying Over $f(x)^m\in\fI A[x]$. This provides $P(X)=\sum_{i=0}a_iX^{i}\in A[X]$
such that $P(x)=0$ and $1\in \gen{a_0,\dots,a_n}\mod {\fI}A[x]$.
 Apply Lemma \ref{Emmanuel}, item \emph{2} with $R=A$. 
\end{proof}
%


\subsubsection*{The induction step}

\begin{proposition}\label{int2}
Let $A$ be a ring with an ideal $\fI$, $B$  an extension of $A$
with $x$ in $B$ such that $B$ is integral over $A[x]$ and $t$ in $B$
such that $xt$ is in $\sqrt{\fI B}$. There exist $b_0,\dots,b_n$
such that $\gen{b_0,\dots,b_n}$ meets
$t^{\NN} + \fI B$ and $b_0,\dots,b_n,b_0x,\dots,b_nx$ are integral
over $A$.
\end{proposition}
\begin{proof}
By Lying Over $xt\in\sqrt{\fI A[x,t]}$ and we can assume as well that $B=A[x,t]$. 

\noindent We apply Proposition \ref{13.7}
with   $R$ the integral closure of $A$ in $B$, $S=R[x,t]=B$, $J=(R[x]:S)$. 
We get an $a\in J\subseteq R[x]$
with $a=t^m+y$, $y\in \fI S$. We have $at=t^{m+1}+y t\in R[x]$.
Since $tx\in \sqrt{\fI S}$, $atx\in R[x]\cap\sqrt{\fI S}$ and by Lying Over
 $\exists e\in\NN,\;(at)^{e}  x^{e}=\sum_i\mu_i x ^{i}$ with $\mu_i$'s 
 in $\fI R$. 
 
\noindent We write $at=p(x)$ with $p(X)\in R[X]$, $q(X)=p(X)^{e}X^{e}-\sum_{i}\mu_i x ^{i}$ written as $\sum_{i=0}^{\ell}a_iX^{i}$ in $R[x][X]$ and $Q(X)=p(X)^{e}X^{e}-\sum_{i}\mu_i X ^{i}\in R[X]$. We have $Q(x)=q(x)=0$ and $\rc_X(p^{e})=\rc_X(q)=\rc_X(Q)
$ $\mod\fI R[x]$. 
 
\noindent Let $R'=R[x]/\fI R[x]$. In $R'$ we have $at=p(x)\in\rc_X(p)$
 and $\sqrt{\rc_X(p)}=\sqrt{\rc_X(Q)}$ by Gauss-Joyal.
 Remark that $t^{m+1}=at-yt$ implies that $t\in\sqrt{\rc_X(p)}+\fI S$.

\noindent If $n$ is a bound for the degree of $Q$, by Lemma \ref{Emmanuel} we get  $b_0,\dots,b_n\in R[x]$ integral over $R$ s.t.
$b_0x,\dots,b_nx$ avec integral over $R$ and $\gen{b_0,\dots,b_n}=\rc_X(Q)$.

\noindent Finally we get $t\in \sqrt{\rc_X(p)}+\fI S=\sqrt{\gen{b_0,\dots,b_n}}+\fI S$
\end{proof}
%

\begin{corollary} \label{corint2}
Let $A$ be a ring with an ideal $\fI$, $B$  an extension of $A$
with $x$ in $B$ such that $B$ is integral over $A[x]$,
$p(X)\in A[X]$ a monic polynomial and $t$ in $B$
such that $p(x)t$ is in $\sqrt{\fI B}$. There exist $b_0,\dots,b_n$
such that $\gen{b_0,\dots,b_n}$ meets
$t^{\NN} + \fI B$ and $b_0,\dots,b_n,b_0x,\dots,b_nx$ are integral
over $A$. 
\end{corollary}
%
\begin{proof}
Let $y=p(x)$, then $x$ is integral over $A[y]$.
Applying Proposition \ref{int2} with $y$ instead of $x$,
we get $b_0,\dots,b_n$
such that $\gen{b_0,\dots,b_n}$ meets
$t^{\NN} + \fI B$ and $b_0,\dots,b_n,b_0y,\dots,b_ny$ are integral
over $A$. We say that this implies $b_jx$'s are integral over $A$.
If the integral dependance of $b_jy$ over $A$ is given by a polynomial of
degree $d$ and $p$ is of degree $m$, multiplying the equation by
$b_j^{(m-1)d}$, one gets an integral dependance equation of $b_jx$
over $A[b_j].$  
\end{proof}
Now we can prove Theorem \ref{main2ter}.


%
\begin{proof} We give the proof for $n=2$, $x_1=x$ and $x_2=y$. \\
The induction from $n-1$ to $n$ follows the same lines as the induction from $1$ to $2$.

\noindent 
First we apply Proposition \ref{3.1} with $A'=A[x]$ instead of $A$, $y$ replacing $x$. We get $s\in 1+\fI B$ with $s$ and $sy$ integral over $A'$.

\noindent   Let $p(X)\in A[X]$ be a monic polynomial such that $p(x)\in \fI B$.
By Lying Over $p(x)$ is integral over $\fI A[x,y]$. We take $t=s^N$ for $N$ big enough such that $tp(x)$ is integral over $\fI A[x,s,sy]$. By Lying Over again
$tp(x)$ is in $\sqrt{\fI A[x,s]}$.

\noindent 
We apply Corollary \ref{corint2} with $A$, $x$, $t$, replacing $B$ by $A[x,s]$.  
We get  $b_0,\dots,b_n\in A[x,s]$
such that $\gen{b_0,\dots,b_n}A[x,s]$ meets
$t^{\NN} + \fI A[x,s]$ and $b_0,\dots,b_n,b_0x,\dots,b_nx$ are integral
over $A$. Since $t\in 1+\fI B$, $1\in \gen{b_0,\dots,b_n}B/\fI B$. As
$B/\fI B$ is finite over $A/\fI$, by Lying Over item \emph{2} we have
that $1=\sum_{i=0}^nb_ig_i(\und b) \mod \fI B$ for some polynomials $g_i$ with coefficients in $A$. 
Let $w=\sum_{i=0}^nb_ig_i(\und b)$. Clearly $w$ and $wx$ are integral over $A$.
Applying lemma \ref{int} with $R=A$, $S=B$ gives $u=w^Ms$ such
that $u$, $ux$ and $uy$ integral over $A$, and we see that $w\in 1+\fI B$. 
\end{proof}
%


\section{Henselian local rings}\label{secHLR}

\noindent \emph{Remark.} Section \ref{subsecThmain3} 
is independant of section \ref{secHLR}.

\subsection{Simple zeroes in commutative rings}

 We consider an arbitrary commutative ring $k$,  $\fI= \mathrm{Rad}(k)$ its Jacobson radical (so\ $1+\fI\subseteq k^{\times}$) and a polynomial
system
$$
f_1(X_1,\dots,X_n) = \cdots = f_n(X_1,\dots,X_n) = 0 ~\eqno{(*)}
$$
which has a simple zero at $(a_1,\dots,a_n)=(\und a)\in k^n$. 
This means 
$$
f_1(\und a)= \cdots = f_n(\und a)= 0\;\hbox{ and }\; J_{\und f}(\und a)\in k^{\times },
$$ 
where $J_{\und f}(\und X)$ is the Jacobian of the system,
i.e.\ the determinant of the Jacobian matrix $\Jac_{\und f}(\und X)=(\partial f_j/\partial X_i)_{1\leq i,j\leq n}$.

Then this zero is unique modulo $\fI= \mathrm{Rad}(k)$ and 
can be isolated in a pure algebraic way as shown by the next lemma.

\begin{lemma} \label{lemZeroSimp} Let us consider the above polynomial system 
$(*)$.\\
Let $L=k[X_1,\dots,X_n]/\gen{f_1,\dots,f_n}=k[x_1,\dots,x_n]$, $S=1+\gen{x_1-a_1,\dots,x_n-a_n}$ and $L_{S}$ the corresponding ``local algebra''.
\begin{enumerate}
\item For $i=1,\dots,n$, $x_i=a_i$ in $L_S$, the natural morphism $k\to L_S$ is an isomorphism. Identifying $k$ with its images in $L$ and $L_S$, we have $L=k\oplus \gen{x_1-a_1,\dots,x_n-a_n}L$ and $L_S=k= L /\gen{x_1-a_1,\dots,x_n-a_n}$.
\item There exists an idempotent $e$ in $S$ such that $ex_i=a_i$ ($i=1,\dots,n$) and $L_S=L[1/e]$.
\item  
$(\und a)$ is the unique zero of $(*)$ equal to $(\und a)$ modulo $\fI$. 
\end{enumerate}
\end{lemma}
%
\begin{proof} 
Making a translation we can replace $(a_1,\dots,a_n)$ by $(0,\dots,0)$.
The evaluation $g\mapsto g(\und 0)$ defined on $k[\und X]$ gives morphisms
$L\to k$ and $L_S\to k$, which we shall note again $g\mapsto g(\und 0)$.
By composing $k\to L_S \to k$ or $k\to L \to k$ we get the identity map.
So $L=k\oplus \gen{x_1,\dots,x_n}L$.

\smallskip\noindent  \emph{1} and \emph{2}. 
After a linear change of variables using $\Jac(\und 0)^{-1}$ we can assume
that $\Jac(\und 0)=\In$, and we write $f_i(\und X)=X_i-g_i(\und X)$
with $g_i(\und X)\in\gen{X_1,\dots,X_n}^2$. So in $L$ we have a matrix $M=M(\und x)\in\mathbb{M}_n(\gen{x_1,\dots,x_n})$ satisfying
$$
\cmatrix{x_1\cr \vdots \cr x_n}= M \cmatrix{x_1\cr \vdots \cr x_n}.
$$
Writing $e(\und x)=\det(\In-M)$ we get $e\in1+\gen{x_1,\dots,x_n}=S$ 
and $ex_i=0$, which implies $eg=eg(\und 0)$ for all $g\in L$. 
In particular  $e^2=e$ and $eh=e$ for $h\in S$, so $L_S=L[1/e]$. Also $x_i=0$ in $L_S$ and $g=g(\und 0)$ for all $g\in L_S$. 

\smallskip\noindent   \emph{3}.  
Let $(y_1,\dots,y_n)$ be a zero with coordinates in $\fI$. 
So we have a $k$-morphism 
$$L
\to k, \;\;g\mapsto g(y_1,\dots,y_n).
$$ 
We can view it as a specialization $x_i\to y_i$.
Item  \emph{2}  gives $e(x_1,\dots,x_n)\in1+\gen{x_1,\dots,x_n}$ with $ex_i=0$. 
Specialising $x_i$ to $y_i$
we obtain $e(y_1,\dots,y_n)y_i=0$ with $e(y_1,\dots,y_n)\in 1+\gen{y_1,\dots,y_n} 
\subseteq 1+\fI \subseteq k^{\times }$.
\end{proof}

\noindent \emph{Remark}. Viewing $L$ as the ring of polynomial functions on the variety defined by the polynomial system $(*)$, the idempotent $e$ 
defines a clopen Zariski subset, it gives two ways of isolating 
the zero $(\und a)$, either by considering the closed subset defined by 
$e=1$ or by considering the open subset defined by making $e$ invertible
(the two subsets are identical).
Moreover point \emph{3} gives a third way of understanding the fact that the zero is isolated: it is the unique zero in the ``infinitesimal neighborhood of
$(\und a)$''.  

\paragraph{Approximate simple zeroes and Newton process}~

\smallskip \noindent 
Here $A$ is a commutative ring with an ideal $\fI$
and we consider a polynomial
system with coefficients in~$A$
$$
f_1(X_1,\dots,X_n) = \cdots = f_n(X_1,\dots,X_n) = 0 ~\eqno{(*)}
$$

\begin{theorem}
\label{thNewtonLin} \emph{(Newton process, see e.g. \cite[Section III-10]{LQPTF})}\\
Let  $(\und a)=(a_1,\ldots ,a_n)\in A^n$ be 
an approximate simple zero of $(*)$ modulo $\fI$:
it gives a zero of $(*)$ in $A/\fI$ and the Jacobian $J_{\und f}(\und a)$ of the system is invertible in $A/\fI$. So the Jacobian matrix $\Jac_{\und f}({\und a})$ is invertible
modu\-lo~$\fI$; let $U(\und a)\in \mathbb{M}_{n}(A)$ be such an inverse modulo $\fI$. Compute 
$$
\cmatrix{b_1\cr \vdots\cr b_n}=\cmatrix{a_1\cr \vdots\cr a_n}-
U(\und a)\cmatrix{f_1(\und a)\cr \vdots\cr f_n(\und a)}.
$$
Then $(b_1, \dots, b_n)$ is a zero of $(*)$ modulo $\fI^2$
and $\Jac(\und b)$ is invertible
modulo~$\fI^2$: one can take $U(\und b)=U(\und a)(2\mathrm{I}_n-\Jac(\und b) U(\und a))$.
\end{theorem}

Let $\fJ$ be the Jacobson radical of the ideal $\fI$, i.e.\ 
the ideal of elements $x$ such that each $y\in1+xA$ is invertible modulo $\fI$.
Then Lemma  \ref{lemZeroSimp}~\emph{3} 
tells us that $(\und a)$ is the unique zero modulo $\fI$
of $(*)$ equal to $(\und a)$ modulo $\fJ$. Since $\fJ$ is also the Jacobson radical of $\fI^2$, $(\und b)$ is the unique zero modulo~$\fI^2$
of $(*)$ which is equal to $(\und b)$ modulo $\fJ$. A fortiori  
$(\und b)$ is the unique zero modulo~$\fI^2$
of $(*)$ which is equal to $(\und a)$ modulo $\fI$.

\medskip \noindent \emph{Remark}. Newton process is used for constructing
a zero of an Hensel system (see Context \ref{contextMHL}) when the Henselian local ring is a ring
of formal power series. Nevertheless, this does not prove that the
coordinates of the zero are inside the Henselization of
the ring generated by the coefficients of the Hensel system. So the
MHL can be seen an improved version of Newton process for the existence of the zero.
On the other hand, Newton process is used in the proof of MHL
(see the proof of Lemma~\ref{lempropMHL}). 

\subsection{Simple  residual zeroes, Henselian rings} \label{subsecSRZHR}

We fix the following context for sections \ref{subsecSRZHR} and \ref{subsecProofMHL}.

\begin{context} \label{contextMHL}
Let $A$ be a local ring with detachable maximal ideal $\fM$,   
and $k=A/\fM$
its  residual field (it is a discrete field).
We consider a polynomial
system
$$
f_1(X_1,\dots,X_n) = \cdots = f_n(X_1,\dots,X_n) = 0 ~\eqno{(*)}
$$
which has a residually simple zero at $(0,\dots,0)$: we have 
$f_i(0,\dots,0) = 0$ residually and  the Jacobian of this system
$J_{\und f}(0,\dots,0)$ is in $A^{\times}$.
In this case we will say that we have a \emph{Hensel system}.
\end{context}

First we remark that if $(C,\fM_C)$ is a local $A$-algebra such that 
the system $(*)$ has a solution $(y_1,\dots,y_n)$ with the $y_i$'s in $\fM_C$, 
then this solution is unique by Lemma \ref{lemZeroSimp}~\emph{3}. 

\smallskip 
To this polynomial system we associate
\[ 
\begin{array}{lll} 
\hbox{the quotient ring}  &&  B=A[X_1,\dots,X_n]/\gen{f_1,\dots,f_n}=A[x_1,\dots,x_n]   \\[1mm] 
\hbox{a maximal ideal of $B$}  &&  \fM_B=\fM +\gen{x_1,\dots,x_n} B \quad (\fM_B\supseteq \fM B)   \\[1mm] 
\hbox{and the local ring}  &&   B_{1+\fM_B}  \hbox{ (usually denoted as } B_{\fM_B}).  
 \end{array}
\]
The ideal $\fM_B$ is maximal because it is the kernel of the morphism
$B\to k$ sending $g(\und x)$ to $\ov g(\und 0)$. This shows also that
 $B/\fM_B=A/\fM$ and hence the natural morphism $A\to B$ is injective.
 So we can identify $A$ with its image in $B$ and we have $B=A\oplus \gen{x_1,\dots,x_n} B$. 
Nevertheless it is not at all evident that the morphism from $A$ to $B_{1+\fM B}$ is injective (this fact will be proved in Corollary \ref{cor2MHL}), so if we speak of $A\subseteq B_{1+\fM B}$ before the proof of Corollary \ref{cor2MHL} is complete, it is an \emph{abus de langage} and it is needed to replace $A$ by its image in
$B_{1+\fM B}$.

It can be easily seen that the natural morphism $\varphi:A\to B_{1+\fM_B}$ satisfies the following \emph{universal property}: 
\\
$\varphi$ is a \emph{local morphism} (i.e., $\varphi(x)\in (B_{1+\fM_B})^{\times }$ implies $x\in A^{\times }$) and for every local morphism $\psi:A\to C$  
such that $(y_1,\dots,y_n)$ is a solution of $(*)$ with the $y_i$'s in the maximal ideal of  $C$,  there exists a unique local morphism $\theta:B\to C$ such that
$\theta\circ \varphi=\psi$.

Since $B_{1+\fM_B}$ satisfies this universal property w.r.t. the system $(*)$
we introduce the notation
$$ B_{1+\fM_B} = \Afn.
$$


The following version of MHL is a kind of ``primitive element theorem''.

\begin{theorem} \label{MHL} \emph{(Multivariate Hensel Lemma)}\\
We consider a Hensel system as in Context \ref{contextMHL} and we use preceeding  notations.\\
Then the local ring $\Afn=B_{1+\fM_B}$ can also be described with only one polynomial equation $f(X)$ such that $f(0)\in\fM$ and $f'(0)$ invertible.
 More precisely there exist

\smallskip \noindent  -- an $y\in \fM_B$,

\smallskip \noindent  --  a monic polynomial $f(X)\in A[X]$
with $f(y)=0$ and $f'(0)\in 1+\fM$ (thus $f'(y) \in  1+\fM_B$), 

\smallskip \noindent such that  

\smallskip \noindent  --  each $x_i$ belongs to $A[y]_{1+\fM+yA[y]}$,

\smallskip \noindent  -- 
the natural morphism
$\Af\to B_{1+\fM_B}$ sending $x$ to $y$ is an isomorphism ($x$ is $X$ viewed in $\Af$).

\smallskip \noindent 
In short $\Afn=A_{\lrb{f}}$. 
\end{theorem}

Before proving Theorem \ref{MHL} we state some corollaries.

A local ring  where each equation of the preceeding form
(a monic polynomial with a simple residual zero) has
a solution residually $0$ is said to be \emph{Henselian}. 

As immediate consequence of the MHL one has the following.
\begin{corollary} \label{corMHL}
Let  $(A,\fm)$ be a Henselian local ring. Assume that a polynomial
system
$(f_1 , \dots , f_n)$ in $A[X_1,\dots,X_n]$
has a residually simple zero at $(0,\dots,0)$. Then the system has a (unique) solution in $A^{n}$ with coordinates in $\fm$.
\end{corollary}

  
\begin{corollary} \label{cor2MHL}
The morphism $A\to \Afn$ is faithfully flat.
In particular it is injective and the divisibility relation is faithfully extended from $A$ to  $\Afn$.
%
\end{corollary}
%
\begin{proof}
It is sufficient to prove the assertions for ${\Af}$ with a monic
polynomial $f$. Since ${\Af}$ is a localization of a free $A$-algebra, it is
flat over $A$. As ${\Af}/\fM_{\Af}=A/\fM$ 
the morphism $A\to {\Af}$
is local,  hence faithfully flat. So for $a,b\in A$, $a$ divides $b$ in $A$ iff
$a$ divides $b$ in ${\Af}$. 
\end{proof}

\hum{ Il y a un pb avec la remarque ci-apr\`es, c'est que l'on ne sait pas \`a quelle d\'efinition du hens\'elis\'e on se r\'ef\`ere. Il faudrait la reformuler.

 \emph{Remark}. Corollary \ref{cor2MHL} gives a key for understanding
the structure of 
the Henselization of a local ring $A$. It is possible to see it
as an inductive limit of extension rings
of the form $\Afn$. 
Indeed, considering two Hensel systems $f_1,\dots,f_n$ and $h_1,\dots,h_p$ we get natural isomorphisms (after a good choice of variables)
$$
{\left(\Afn\right)}_{\lrb{h_1,\dots,h_p}} \to 
A_{\lrb{f_1,\dots,f_n,h_1,\dots,h_p}} \to \left(A_{\lrb{h_1,\dots,h_p}}\right)_{\lrb{f_1,\dots,f_n}}.
$$
Moreover Corollary \ref{cor2MHL} tells us that the successive morphisms
$A\to \Afn \to A_{\lrb{f_1,\dots,f_n,h_1,\dots,h_p}}$
are injective, so the inductive limit can be viewed as an union. 
}

\subsection{Proof of the Multivariate Hensel Lemma} \label{subsecProofMHL}

\smallskip \noindent 
 We begin by a slight transformation or our polynomial system in order to 
 being able to get the hypotheses of ZMT for the ring associated to the new system.

\begin{proposition} \label{propNewsystem}
Let a polynomial
system
$$
f_1(X_1,\dots,X_n) = \cdots = f_n(X_1,\dots,X_n) = 0 ~\eqno{(*)}
$$
which has a residually simple zero at $(0,\dots,0)$. We use preceeding notations for $B$ and $\fM_B$.

\noindent One can find $f_{n+1}(X_1,\dots,X_n, X_{n+1})\in A[X_1,\dots,X_{n+1}]$ such that
for the new system 
$$
f_1(X_1,\dots,X_n) = \cdots = f_n(X_1,\dots,X_n)= f_{n+1}(X_1,\dots,X_n, X_{n+1}) = 0 ~\eqno{(**)}
$$
we have again $f_{n+1}(0,\dots,0)\in\fM$, with Jacobian $J'(0,\dots,0)$ invertible 
and if we call 
$$
B'=A[x_1,\dots,x_n, x_{n+1}]=A[X_1,\dots,X_n, X_{n+1}]/\gen{f_1,\dots,f_{n+1}}
$$
then $x_1,\dots,x_n, x_{n+1}\in\fM B'$ (this means $\fM B'=\fM_{B'}$), and the natural morphism $B_{\fM_B}\to B'_{\fM B'}$ is an isomorphism. 

\noindent  In short with the new system we have
 $x_1,\dots, x_{n+1}\in\fM A[x_1,\dots, x_{n+1}]$
and
   $\Afn=A_{\lrb{f_1,\dots,f_{n+1}}} $.  
\end{proposition}
%
\begin{proof}
Applying Lemma \ref{lemZeroSimp} to the residual system we get $e(X_1,\dots,X_n)$ such that in $k[x_1,\dots,x_n]$ we have $e^2=e$, $ex_i=0$.
So if we consider the localization $B[1/e]$ we get residually $k[x_1,\dots,x_n,1/e]=k$, more precisely $e=1$ and $x_i=0$ in $k[x_1,\dots,x_n,1/e]$. 
In other words if we
introduce a new variable $T$ and the equation $Te(X_1,\dots,X_n)=1$
we get a new polynomial system which has residually only one zero $(0,\dots,0,1)$.
In order to get a Hensel system we introduce the variable $X_{n+1}$ ($=1-T$) with
the equation $1-(1-X_{n+1})e(X_1,\dots,X_n)$, and $(0,\dots,0)$ is 
the unique residual zero. Moreover
if we call $J'(x_1,\dots,x_{n+1})$ the Jacobian of the new system in $B'$
then $J'(x_1,\dots,x_{n+1})=J_{\und f}(x_1,\dots,x_{n})e(x_1,\dots,x_{n})$
and $J'(0,\dots,0)=J_{\und f}(0,\dots,0) \mod \fM$ is invertible.

\smallskip 
\noindent NB: Let us note that there is a little abuse of notations:
we have $e(x_1,\dots,x_n)=1$ in $k[x_1,\dots,x_{n+1}]=B'/\fM_{B'}$ but in general
$e(x_1,\dots,x_n)\neq 1$ in $k[x_1,\dots,x_{n}]=B/\fM_B$, meaning that
the morphism $k[x_1,\dots,x_{n}]\to k[x_1,\dots,x_{n+1}]$ is not injective.
It would be necessary to change the names of the $x_i$'s when changing the ring!
\end{proof}

\Grandcadre{In the following we assume w.l.o.g. that the system $(*)$ satisfies $x_1,\dots,x_n\in\fM B$.}  

Applying Theorem \ref{main2bis} to $B=A[x_1,\dots,x_n]$, 
$\fM\subseteq A$ and $x_1,\dots,x_n\in\fM B$ (so $B/\fM B=A/\fM$)
we get an $s\in1+\fM B$ such that $s,sx_1,\dots,sx_n$ are integral over $A$. So $s=S(x_1,\dots,x_n)$, where $S\in 1 +\fM A[X_1,\dots,X_n]$).
We are going to prove the following proposition, which clearly implies Theorem~\ref{MHL} if the given polynomial $f$ is monic.

\begin{proposition} \label{propMHL}
We can construct a polynomial $h(T)\in A[T]$ such that $h(s)=0$, $h'(s)\in 1+\fM B$ (more precisely, $h(T)=T^{N}(T-1)$ modulo $\fM A[T]$), the $x_i$ are expressed as rational fractions in $s$ with denominator in $1+\fM+(s-1)A[s]$, and letting $f(X)=h(1+X)$, the natural morphism ${\Af}\to B_{1+\fM B}=A_{\lrb{f_1,\dots,f_n}}$ sending $x$ to $s-1$
is an isomorphism. 
\end{proposition}
%
\begin{proof} Let $D = A[s,sx_1,\dots,sx_n]$ and therefore we have $A\subseteq A[s]\subseteq D\subseteq B$. 
We call $m_0=1$, $m_1$, \dots, $m_\ell$  monomials in the $(sx_i)$'s such that $m_0,\dots,m_\ell$
generate $D$ as an $A[s]$-module and we can also assume that $m_i=sx_i$ for $i=1,\dots,n$.
We have $D=A[s]+m_1D+\cdots+m_\ell D$.

\noindent 
Let $\fM_D=\fM B\cap D$ and $\fM_{A[s]}=\fM B \cap A[s]$.

\noindent  Since $m_1,\dots,m_\ell\in\fM_D$ we have $\fM_D=\fM_{A[s]}+m_1D+\cdots+m_\ell D$. 

\noindent As $s-1\in\fM_{A[s]}$ and $A[s]=A[s-1]=A+(s-1)A[s]$ we have $\fM_{A[s]}=\fM+(s-1)A[s]$.

\noindent Notice that  for all $v\in B$ there exists an exponent $r$ such that $s^rv\in D$. Moreover if $v\in \fM B$, there exists an exponent $r$
such that $s^{r}v\in \fM D$. In particular, since $m_j\in \fM B$  there exists an exponent $r_0$
such that 
all $s^{r_0}m_j\in \fM D$ ($j= 1,\dots,\ell$).

\noindent We write this fact as 
$$
s^{r_0}m_j=\left(\som_{i=1}^{\ell}\mu_{ij}(s)m_i\right)+\mu_{0j}(s)
$$
where $\mu_{ij}(s)\in\fM A[s]$ for all $i,j$.

\noindent Let $M(s)=(\mu_{ij}(s))_{1\leq i,j\leq \ell}$.
We have then 
$$
s^{r_0} \cmatrix{m_1\cr\vdots\cr m_\ell}= M(s) \cmatrix{m_1\cr\vdots\cr m_\ell}
+ \cmatrix{\mu_{01}(s)\cr\vdots\cr \mu_{0\ell}(s)}.
$$
Let $d(T)=\det (T^{r_0}\,\mathrm{I}_\ell-M(T))$, multiplying by the adjoint matrix $P(s)$
we get 
$$
d(s)\cmatrix{m_1\cr\vdots\cr m_\ell}=P(s)
\cmatrix{\mu_{01}(s)\cr\vdots\cr \mu_{0\ell}(s)} \in 
\mathbb{M}_{\ell,1}(\fM A[s]).
$$

\smallskip \noindent \emph{Summing up.} 
We have  found a polynomial $d(T)\in A[T]$ such that:

\smallskip \noindent \hspace*{1cm} i) $d(T)=T^N$ modulo $\fM A[T]$ for some $N$,
and so $d(s)s^r\in 1+ \fM_{A[s]}$ for all $r\geq 0$,

\smallskip \noindent \hspace*{1cm} ii)  
one has  $d(s) m_j=\nu_j(s)\in \fM A[s]$,
this implies  $d(s) \fM_D\subseteq  \fM_{A[s]}$,

\smallskip \noindent \hspace*{1cm} iii)  
given an arbitrary $v\in \fM B$  one has
an exponent $r$ such that $s^{r}d(s) v\in \fM A[s]$,

\smallskip \noindent  Let $q$ be an exponent such that 
$s^{q}d(s) (s-1)\in \fM A[s]$ and let us define
$h(T)=d(T)T^q(T-1)-\mu(T)\in A[T]$.
So $h(s)=0$ and $h(T)=T^{N+q}(T-1)$ modulo $\fM A[T]$. 
Notice that $h'(1)\in1+\fM$,
which implies that $s$ is a root of $h(T)$ which is residually simple.

\smallskip \noindent  Now we finish the proof of Proposition \ref{propMHL} using the following  general lemma.
\end{proof}
%

\begin{lemma} \label{lempropMHL} 
Let $h(T)\in  A[T]$ such that  $h(s)=0$ and $h'(1)\in1+\fM$.\\
Let us set $f(X)=h(1+X)$, $A[x]=A[X]/\gen{f(X)}$, $t=1+x$, $\Af=A[x]_{1+\fM+ x A[x]}$.\\
Let  $\theta:{\Af}\to B_{1+\fM B}=\Afn$ be the  natural morphism
 sending $x$ to $s-1$ (and $t$ to $s$) given by the universal property of $\Af$.\\
Then  $\theta$ is in fact an isomorphism.
\end{lemma}
%
\begin{proof}
In order to prove that $\theta$ is an isomorphism,
it is sufficient to find a zero $(z_1,\dots,z_n)$ of the system $(*)$ in ${\Af}$
with coordinates $z_i$ in the maximal, and such that

\smallskip \noindent\hspace*{.5cm} --  $\theta(z_i)=x_i$ for each $i$, 

\smallskip \noindent \hspace*{.5cm} --  the natural morphism
$B_{1+\fM B}\to {\Af}$ sending $(x_1,\dots,x_n)$ to $(z_1,\dots,z_n)$ 
sends $s$ to $t$.

\smallskip \noindent We have $d(s)sx_i=\nu_i(s)$ where $\nu_i(T)\in \fM A[T]$.
We let $q(T)=Td(T)$.

\noindent We have $q(T)\in T^{N+1}+\fM A[T]$,  $t\in 1+\fM\Af$ and
$q(t)\in 1+\fM\Af$. We let $$z_i=\nu_i(t)/q(t)\in \fM {\Af}$$ 
and we get $\theta(z_i)=\nu_i(s)/q(s)=x_i$ for each $i$.

\noindent Let 
 $$
 \fI=\gen{f_1(z_1,\dots,z_n),\dots,f_n(z_1,\dots,z_n)}
 $$ 
as ideal of $ {\Af}$. We get $\fI\subseteq \fM \Af$.

\smallskip \noindent 
We are going to show that $\fI=\fI^2$, so $\fI\subseteq \fI\ \fM'$, where $\fM'=\fM\Af+x\Af$ is the maximal ideal of $\Af$. This implies $\fI=0$ by Nakayama's Lemma.
This will show that $(z_1,\dots,z_n)$ is a zero of $(*)$ with coordinates in $\fM'$.

\noindent 
  By Newton process we can construct a  zero modulo~$\fI^2$, 
let us call it $(y_1,\dots,y_n)$.
The system $(*)$ has the zero  $(y_1,\dots,y_n)$ residually null in the local 
ring $\Af/\fI^2$. 
By the universal property of $\Afn$ there is a morphism $\lambda:\Afn\to \Af/\fI^2$ sending $x_i$ to $y_i$.
\\
We let $y=S(y_1,\dots,y_n)$, so $\lambda(s)=y\mod \fI^2$,
$h(y)=\lambda(h(s))\mod \fI^2$, i.e. $h(y)= 0\mod \fI^2.$

\noindent 
Since $h(t)=0$, $h'(t)\in\Af^{\times}$, $t=y=1 \mod \fM'$ we write 
$$h(y)=h(t)+(t-y)(h'(t)+ (t-y)h_1(t,y)),$$
 we have  $h'(t)+ (t-y)h_1(t,y)\in\Af^{\times}$  and we get $t-y\in \fI^2$.
 
\noindent We have $0=\lambda(q(s)x_i-\nu_i(s))=q(y)y_i-\nu_i(y)$ in $\Af/\fI^2$,
$q(y)y_i-q(t)y_i\in\fI^2$ and $\nu_i(t)-\nu_i(y)\in\fI^2$, so
$q(t)y_i-\nu_i(t)\in\fI^2$, i.e.\ $y_i=z_i\mod \fI^2$.
Finally 
$$
0=\lambda(f_i(x_1,\dots,x_n))=f_i(y_1,\dots,y_n)=f_i(z_1,\dots,z_n)\mod \fI^2.
$$
This shows that $\fI\subseteq \fI^2$, so $\fI=0$. Now, since $(z_1,\dots,z_n)$
is a zero of $(*)$ residually null in $\Af$, by the universal property of 
$\Afn$ we can see  $\lambda$ 
as a morphism from  $\Afn$ to $\Af$ 
sending $x_i$ \hbox{to $y_i=z_i$}. 

\smallskip \noindent  Finally, we show that $\lambda(s)=t$. This follows from
$h(\lambda(s))=\lambda(h(s))=0$ and $s\in1+\fM A[x_1,\dots,x_n]$ in $B$
which implies $\lambda(s)\in1+\fM A[y_1,\dots,y_n]\subseteq 1+\fM\Af$
so $h'(\lambda(s)) \in 1+\fM\Af$. 
\end{proof}

In order to get Theorem \ref{MHL} from Proposition \ref{propMHL}
we use the following lemma.

\begin{lemma} \label{lemMonicMHL2} \emph{(see \cite[Lemma 5.3]{ALP})}\\
 Let $(A,\fM)$ be a local ring,
  \(f(X)=a_nX^n+\cdots +a_1 X+a_0\), with
\(a_1\in A^\times\) and \(a_0\in\fM\). There exists a monic polynomial
\(g(X)\in A[X]\), \(g(X)=X^n+\cdots+b_1 X+b_0\), with \(b_1\in
A^\times\) and \(b_0\in\fM\), such that the following equality holds in
\(A(X)\) (the Nagata localization of $A[X]$):
\[a_0\cdot g(X)=(X+1)^n f\left( {-a_0 a_1^{-1}\over X+1}
\right).\] 
Moreover $\Af$ is isomorphic to $A_{\lrb g}$.
\end{lemma}
\begin{proof}
We have 
\[\begin{array}{rl}
X^n f\left(\frac{-a_{0}
a_{1}^{-1}}{X}\right) & = a_{0}\cdot \left(X^n-X^{n-1}+a_{0}
\sum_{j=2}^{n}(-1)^j a_{j}a_{0}^{j-2} a_{1}^{-j} 
X^{n-j}\right)\\[2mm]
 & = a_{0} h(X)
\end{array}\]
with
\[ h(X)= X^n-X^{n-1}+a_{0}
\sum_{j=2}^{n}(-1)^j   a_{j}a_{0}^{j-2} a_{1}^{-j}  X^{n-j}
= X^n-X^{n-1}+a_{0} \ell(X) \]
We let \(g(X)=h(X+1)=X^n+\cdots+b_1 X+b_0\). It is a monic polynomial,
with constant term \(b_0=g(0)=h(1)=a_{0} \ell(1)\in\fM\), and linear
term \(b_1=g'(0)=h'(1)=1+a_{0}\ell'(1)\in1+\fM\).
\end{proof}

\subsection{An example of Multivariate Hensel Lemma}
\label{subsecMHLexample} 

 In this section we analyse an example where $A$ is 
the local ring $\QQ[a,b]_S$, $S$ being the monoid
of elements $p(a,b)\in\QQ[a,b]$ such that $p(0,0)\neq 0$.
We take next $B = A[x,y]$ where $x,y$ are defined by the equations
$$-a + x   + bxy + 2bx^2      = 0,~~~~~~ -b + y +ax^2 + axy + by^2 = 0$$
We shall compute $s\in B$ integral over $A$ such that $sx,sy$ integral
over $B$ and $s=1$ mod. $\fM B$.

 Following the proof we apply Proposition \ref{3.1} and we take $t = 1+ax+by$. We have that
$t=1$ mod. $\fM B$ and $t,ty$ integral over $A[x]$. We have even
$ty = y +axy +by^2 = b-ax^2$ in $A[x]$. The equation for $t$ is
$$
t^2 - (1+ax) t -b +ax^2
$$
We have then
$$
tx = x+ax^2 + bxy = a + (a-2b)x^2.
$$
Notice that we are now in the situation of the proof of Proposition
\ref{lemDtrivial} with $Q(X,T)=TX-(a + (a-2b)X^2)$.
Since $Q$ has degree $1$ we get without extra work 
and so
$$
(t -(a-2b)x)x = a
$$
If we take $w = t-(a-2b)x = 1+2bx +by$ we have $w=1$ mod. $\fM B$ and $wx$ in $A$ and $w$ is integral over $A$. Indeed $w$ is integral
over $A[1/w]$ since $x$ is in $A[1/w]$ and
$w$ is integral over $A[x]$. 

 If we take $u = tw^2$ we have $u, ux, uy$ integral over $A$. Indeed, $wx$ is in $A$ and
since $t^2 - (1+ax) t -b +ax^2=0$ we have $tw$ and hence $u$ integral over $A$.
Since $ty = b-ax^2$ we have $uy = bw^2 -a(wx)^2$ integral over $A$. Finally
$ux = (tw)(wx)$ is integral over $A$.

It can be checked that $u$ is a  root of a monic polynomial $f$ of degree 4
of the form $U^3(U-1)$ residually.
\[
\begin{array}{ccc}
-u^4+ (1+4\,a\,b+a^2+3\,b^2)\,u^3\\[1mm]
+ b\,(b^5+8\,a\,b^4+7\,a^2\,b^3-a^3\,b^2-4\,b\,a^4
+a^5-6\,a^2\,b-a^3+4\,a\,b^2)\,u^2\\[1mm] 
-a^2\,b^2\,(a-b)\,(a+2\,b)\,(2\,b^2-9\,a\,b+a^2)\,u 
+a^4\,b^3\,(a-4\,b)\,(a+2\,b)^2\,(a-b)^2=0  
\end{array}
\]

\hum{~
\begin{enumerate}
\item Donner le polynome $f$
\item $B$ n'est pas finie sur $A$
\item $B[1/s]$ est finie sur $A$
\item $B[1/s]$ est \'egale \`a  $\Af$?
\end{enumerate}
 }

\section{Structure of quasi finite algebras}
\label{subsecThmain3} 


Let us recall that in classical mathematics
an $A$-algebra $B$ is said to be quasi-finite if it
is of finite type and if prime ideals of $B$ lying over any prime ideal of $A$ are incomparable.

This last requirement means that the morphism $A\to B$ is zero-dimensional.
A constructive characterization of zero-dimensional morphisms uses 
the zero-dimensional reduced ring $A^{\bullet}$ generated by $A$.

A zero-dimensional reduced ring is characterized by the fact that every element $a$ possesses a quasi inverse: an element $b$ such that $a^2b=a$
and $b^2a=b$. Such a ring is also said to be Von Neuman regular or absolutely flat.
The element $ab$ is an idemptent $e_a$. In the component $A[1/e_a]$, $a$ is invertible, and  $a=0$ in the other component $A/\gen{e_a}$.

From an algorithmic point of view this implies that algorithms for discrete 
fields are easily transformed in algorithms for zero-dimensional reduced rings
(for more details see \cite[Chapter 4]{LQPTF}).

The ring $A^{\bullet}$ can be obtained as a direct limit of rings
$$
 A[a_1\bul,a_2\bul,\ldots,a_n\bul]
\simeq
\left( A[T_1,T_2,\ldots,T_n]/{\fa}\right)_{\mathrm{red}}
$$
with $\fa=\gen{(a_iT_i^2-T_i)_{i=1}^n,(T_ia_i^2-a_i)_{i=1}^n}$
(for more details see \cite[section 11.4]{LQPTF}).
The direct limit is along the p.o.\ set of finite sequences of elements of $A$,
ordered by $(a_1,\dots,a_n)\preceq(b_1,\dots,b_m)$ iff one has (for each $i$) $b_{k_i}=a_i$
for some map $\{1,\dots,n\}\vers k \{1,\dots,m\}$. 

\smallskip 
In classical mathematics we obtain the following equivalence.

\begin{proposition} \label{propdefizerodim}
Let  $\varphi:A \to B$ a morphism of commutative rings.
\begin{enumerate}
\item Prime ideals of $B$ lying over any prime ideal of $A$ are incomparable.  
\item The ring $A\bul\otimes_AB$ is a zero-dimensional ring. 
\end{enumerate}
\end{proposition}

The morphism $A\to B$ is  not required to be injective, but the proposition
involves only the
structure of $B$ as $\varphi(A)$-algebra.

\emph{The second item is taken to be the correct definition
of zero-dimensional morphisms in constructive mathematics}.

This gives also a good definition of \emph{quasi-finite morphisms} in constructive mathematics: indeed a \emph{quasi-finite $A$-algebra} is an algebra $B$ of finite type
such that the structure morphism $A\to B$ is zero-dimensional.

We have the following concrete characterization of zero-dimensional morphisms for algebras of finite type. 

\begin{proposition} \label{propdefiquasifini}
Let  $B$ be an $A$-algebra of finite type. The following are equivalent.
\begin{enumerate}
\item The structure map $A\to B$ is a zero dimensional morphism.  
\item There exist $a_1,\dots,a_p\in A$ such that for each $I\subseteq \{1,\dots,p\}$, if we let $I'=\{1,\dots,p\}\setminus I$, 
$\fa_{\und a,I}=\gen{a_i,\,i\in I} $, $\alpha_{\und a,I'}=
\prod_{i\in I'}a_i$ and 
$ 
A_{(\und a,I)}=\left(A/\fa_{\und a,I} \right) \crac 1 {\alpha_{\und a,I'}}   
 $
then the ring $B_{ (\und a,I) }$ is integral over $A_{ (\und a,I) }$. 

\end{enumerate}
\end{proposition}

Let us insist here on the fact that the equivalence in Proposition 
\ref{propdefiquasifini} has a constructive
proof.

\begin{theorem}\label{main3} \emph{(ZMT \`a la Raynaud, \cite{Ray})}\\
Let $A\subseteq B = A[x_1,\dots,x_n]$ be rings  
 such that the inclusion morphism $A\to B$ is zero dimensional
 (in other words, $B$ is quasi-finite over $A$).
Let $C$ be the integral closure of $A$ in $B$.
Then  there exist elements $s_1,\dots,s_m$ in $C$, comaximal in $B$, such that
all  $s_ix_j\in C$. \\
In particular for each $i$, $C[1/s_i]=B[1/s_i]$. 
Moreover  letting $C'=A[(s_i),(s_ix_j)],$ which is finite over $A$,
 we get also  $C'[1/s_i]=B[1/s_i]$ for each $i$.
 \end{theorem}
\begin{proof}
The concrete hypothesis is item \emph{2.} in Proposition \ref{propdefiquasifini}.
We have to find elements $s_1,\dots,s_m$ integral over $A$, comaximal in $B$,
such that all $s_ix_j$ are integral over $A$.

\noindent 
The proof is by induction on $p$, the case $p=0$ being trivial (in this case $B$
is finite over $A$ by hypothesis). 

\noindent 
Assume we have the conclusion for $p-1$ and let $a=a_p$. The induction hypothesis is applied to the morphisms $A/aA\to B/aB$ and $A[1/a]\to B[1/a]$.

\noindent 
First we get $s_1,\dots,s_m$ integral over $A/aA$, comaximal in $B/aB$
with all $s_ix_j$ integral over $A/aA$. 
Let $B'=A[(s_i),(s_ix_j)]$ ($1\leq i\leq m $, $1\leq j\leq n$).
Applying Theorem  \ref{main2bis} to $A\subseteq B'$ and $\fI=aA$ we obtain $w\in 1+aB'$ such that all $ws_i$'s and $ws_ix_j$'s are integral over $A$.

\noindent 
Second, we get $t_1,\dots,t_q$ integral over $A[1/a]$, comaximal in $B[1/a]$
with all $t_ix_j$ integral over $A[1/a]$. This gives, for $N$ big enough, $a^N\in\gen{t_1,\dots,t_q}B$
and all $a^Nt_i $'s and $a^Nt_ix_j $'s integral over $A$.

\noindent Since $1\in\gen{s_1,\dots,s_m,a}B$ and $1\in \gen{w,a}B$, we have
$$
   1\in\gen{ws_1,\dots,ws_m,a^{2N}}B\subseteq
   \gen{ws_1,\dots,ws_m,a^Nt_1,\dots,a^Nt_q} .
$$ 
So we have our conclusion with 
the family $(ws_1,\dots,ws_m,a^Nt_1,\dots,a^Nt_q)$.
\end{proof}



\begin{thebibliography}{50}
\addcontentsline{toc}{section}{References}

\bibitem{ALP}
{\sc M. Alonso, H. Lombardi, H. Perdry}.
\newblock{{\em Elementary Constructive Theory of Henselian Local Rings.}} 
\newblock{Math. Logic Quarterly {\bf 54} (3),  (2008), 253--271.}

\bibitem{Atiyah}
{\sc M. Atiyah, L. MacDonald}.
\newblock{{\em Introduction to Commutative Algebra.}}
\newblock{Addison Wesley series in Mathematics, (1969).}

\bibitem{BPR}
{\sc S. Basu, R. Pollack \& M.-F. Roy }{\em Algorithms in Real Algebraic Geometry}. Springer-Verlag.  2nd edition  (2006).


\bibitem{BB} {\sc E. {Bishop}, D. {Bridges}}.  
\newblock
{\em 
Constructive Analysis.  }
\newblock{Springer-Verlag (1985).}

\bibitem{BR} {\sc D. {Bridges}, F. {Richman}}.	
\newblock
{\em   Varieties of Constructive Mathematics.}
\newblock{ 
London Math. Soc. LNS 97. Cambridge University Press (1987).}

\bibitem{coq:seminormal} {\sc T. Coquand}. 
\newblock {\em On seminormality.} %
\newblock J. Algebra, {\bf 305}, (2006), 577--584.

\bibitem{coq:valspace} {\sc T. Coquand}. \newblock {\em Space of valuations.} %
\newblock Ann.~Pure Appl.~Logic, {\bf 157}, (2009), 97--109.

\bibitem{coq:generating} {\sc T. Coquand, H. Lombardi, C. Quitt\'e}. 
\newblock {\em Generating non noetherian modules constructively.} %
\newblock Manuscripta Math., {\bf 115}, (2004), 513--520.

\bibitem{CQ2012} {\sc T. Coquand, C. Quitt\'e}. 
\newblock {\em Constructive finite free resolutions.} %
\newblock Manuscripta Math., {\bf 137}, (2012), 331--345.



\bibitem{CLS} {\sc T. Coquand, H. Lombardi, P. Schuster} 
\newblock {\em Spectral schemes as ringed lattices.} %
\newblock Ann. Math. Artif. Intell., {\bf 56}, (2009), 339--360.


\bibitem{DLQS}
{\sc L. Ducos, H. Lombardi, C. Quitt\'e , M. Salou}.
\newblock
{\em Th\'eorie algorithmique des anneaux arithm\'etiques, de Pr\"ufer et de 
Dedekind.}
\newblock{Journal of Algebra. {\bf 281}, (2004), 604--650.}


\bibitem{Hallouin}
{\sc E. Hallouin}.
\newblock{{\em Parcours initiatique \`a travers la th\'eorie des valuations.}}
\newblock{technical report. Universit\'e de Poitiers, (1997).}

%

\bibitem{EGA4}
{\sc A. Grothendieck}.
\newblock{\em \'El\'ements de G\'eom\'etrie Alg\'ebrique IV.}
\newblock{}



\bibitem{LQPTF}
{\sc H. Lombardi, C. Quitt\'e}.
{\em Commutative Algebra, Constructive Methods.} Algebra and Applications, Vol. 20. Berlin, New-York: Springer. (Revised and expanded.) Translated from the French by Tania K. Roblot. {\em Alg\`ebre Commutative, M\'ethodes Constructives.}
{Calvage \& Mounet, (2011).}

\bibitem{Ray}  {\sc M. Raynaud}. 
\newblock{{\em Anneaux locaux hens\'eliens.}}
\newblock{Springer Lecture Notes in Mathematics No 169, (1970).}


\bibitem{Richman}
{\sc R. Mines, F. Richman and W. Ruitenburg}.
\newblock{{\em A Course in Constructive Algebra.}}
\newblock{Springer-Verlag, (1988)}

\bibitem{Peskine}
{\sc C. Peskine}.
\newblock{\em Une g\'en\'eralization du ``main theorem'' de Zariski.  }
\newblock{Bull. Sci. Math. (2)  90  (1966) 119--127.}


\bibitem{PeskineBook} {\sc C. Peskine}. {\em An Algebraic Introduction to Complex Projective Geometry: Commutative Algebra}.
Cambridge University Press, (1996).






\bibitem{yengui:maximal} 
{\sc I. Yengui}. 
\newblock {\em Making the use of maximal ideals constructive.} 
\newblock Theoret. Comput. Sci., {\bf 392},
(2008), 174--178.



\bibitem{Zar2}
{\sc O. Zariski.}
\newblock{\em Analytical irreducibility of normal varieties.} 
\newblock{Annals of Math. 49  (1948) 352--361.}



\end{thebibliography}
\end{document}